\documentclass[11pt]{amsart}
\usepackage{amssymb,latexsym}
\newtheorem{theorem}{Theorem}[section]

\newtheorem{lemma}[theorem]{Lemma}
\newtheorem{remark}[theorem]{Remark}

\begin{document}

\title{Local Gromov-Witten invariants of blowups of  Fano surfaces}

\author[Jianxun Hu$^1$]{Jianxun Hu$^1$  }
\address{Department of Mathematics\\ Zhongshan
University\\Guangzhou\\ P. R. China}
\email{stsjxhu@mail.sysu.edu.cn}
\thanks{${}^1$Partially supported by the NSFC Grant (10631050 and 10825105), NKBRPC (2006CB805905) }

\maketitle
\begin{abstract}
   In this paper, using the degeneration formula we obtain a
   blowup formulae of local Gromov-Witten invariants of
   Fano surfaces. This formula makes it possible to compute the
   local Gromov-Witten invariants of non-toric Fano surfaces from
   toric Fano surface, such as del Pezzo surfaces. This formula
   also verified an expectation of Chiang-Klemm-Yau-Zaslow in the section 8.3 of \cite{CKYZ}.
\end{abstract}

\tableofcontents

\section{Introduction}
Local del Pezzo surface used to play an important role in physics.
Local de Pezzo surfaces are usually associated to phase
transitions in the K\"ahler moduli space of various string,
M-theory, and F-theory compactifications. More precisely, del
Pezzo contractions in Calabi-Yau threefolds are related to quantum
field theories in four and five dimensions \cite{DKV, KKV, MS} via
geometric engineering. Non-toric del Pezzo surfaces seem to be
related to exotic physics in four, five and six dimensions such as
nontrivial fixed points of the renormalization group \cite{GMS}
without lagrangian description and strongly interacting
noncritical strings. There is also a relation between non-toric
del Pezzo surfaces and string junctions in F-theory
\cite{KMV,LMW}. Certain problems of physical interest such as
counting of BPS states reduce to questions related to topological
strings on local del Pezzo surfaces.

``Local mirror symmetry" mathematically refers to a specialization
of mirror symmetry techniques to address the geometry of Fano
surfaces within Calabi-Yau manifolds. The procedure produces
certain ``invariants" associated to surfaces.

Let $S$ be a Fano surface and $K_S$ its canonical bundle. For
$\beta\in H_2(S, \mathbb{Z})$, denote by $\overline{\mathcal
M}_{g,k}(S,\beta)$ the moduli space of $k$-pointed stable maps of
degree $\beta$ to $S$. Then the following diagram
\begin{equation*}
  \begin{array}{ccc}
   \overline{\mathcal M}_{g,1}(S,\beta) &
   \stackrel{ev}{\longrightarrow} & S\\
    &  &  \\
    \rho\downarrow &  &  \\
    &   &  \\
    \overline{\mathcal M}_{g,0}(S,\beta) & &
    \end{array}
\end{equation*}
defines the obstruction bundle $R^1\rho_*ev^*K_S$ whose fiber over
a stable map $f:C\longrightarrow S$ is given by $H^1(C, f^*K_S)$.

One can define the local Gromov-Witten invariants \cite{CKYZ} of
$K_S$ by
\begin{equation}\label{local-definition}
  K_{g,\beta}^S = \int_{[\overline{\mathcal M}_{g,0}(S,\beta)]^{vir}}
  e(R^1\rho_*ev^*K_S).
\end{equation}

Yang-Zhou \cite{YZ} generalized this definition of local
Gromov-Witten invariants to the canonical line bundles of toric
surfaces, not necessarily Fano. When $S$ is toric, one can use the
localization technique to compute the local Gromov-Witten
invariants, see \cite{CKYZ, KZ, YZ}. When $S$ is non-toric Fano,
the localization technique is no longer valid. Therefore few
results on local Gromov-Witten invariants of non-toric Fano
surface are known. In this paper, we will study how to compute the
local Gromov-Witten invariants of some non-toric surfaces from
that of some toric surfaces.

Denote by $Y_S = \mathbb{P}(K_S\oplus {\mathcal O})$ the
projective bundle completion of the total space of the canonical
bundle $K_S$. Then $Y_S$ has two canonical sections $S^+, S^-$
with normal bundle $N_{S^+|Y_S}\cong K_S$ and $N_{S^-|Y_S}\cong -
K_S$ respectively. Then the normal bundle of a surface $S^+$
inside $Y_S$ is negative. If the image of a stable map lies in
$S^+$, it is not able to deform it outside of $S^+$. This means if
we denote also by $\beta$ the image of a class $\beta\in
H_2(S,\mathbb{Z})$ under the inclusion map $S\hookrightarrow Y_S$
via the section $S^+$, then one has $\overline{\mathcal
M}_{g,0}(Y_S,\beta) = \overline{\mathcal M}_{g,0}(S,\beta)$. By
the constructions of the virtual fundamental cycles, we have
\begin{equation}\label{virtual-cycle}
[\overline{\mathcal M}_{g,0}(Y_S,\beta)]^{vir} =
[\overline{\mathcal M}_{g,0}(S:\beta)]^{vir}\cap
e(R^1\rho_*ev^*K_S).
\end{equation}
Denote the Gromov-Witten invariant of $Y_S$ of degree $\beta$ by
\begin{equation}
   n^{Y_S}_{g,\beta} = \int_{[\overline{\mathcal
   M}_{g,0}(Y_S,\beta)]^{vir}}1.
\end{equation}
Therefore, from (\ref{local-definition}) and
(\ref{virtual-cycle}), we have
\begin{equation}\label{local=GW}
    K^S_{g,\beta} = n^{Y_S}_{g,\beta}.
\end{equation}

Denote by $p:\tilde{S}\longrightarrow S$ the natural projection of
the blow-up of $S$ at a smooth point $p_0\in S$. Let $\beta\in
H_2(S,\mathbb{Z})$ and $p!(\beta) = PDp^*PD(\beta)\in
H_2(\tilde{S}, \mathbb{Z})$. In \cite{CKYZ}, the authors computed
the genus zero local Gromov-Witten invariants of $\mathbb{P}^2$
and the Hirzebruch surface $\mathbb{F}_1$ via the localization
technique in the case of lower degrees. They observed that the
genus zero local Gromov-Witten invariants of $K_{\mathbb{P}^2}$ of
degree $\beta$ are equal to the genus zero local Gromov-Witten
invariants of $K_{\mathbb{F}_1}$ of degree $p!(\beta)$. In this
paper, we use the degeneration formula to study the change of
local Gromov-Witten invariants under the blowup of the Fano
surfaces and verify their observation and generalize it to any
genus case. Our main theorem is

\begin{theorem}\label{thm-1-1}
Suppose that $S$ is a Fano surface and  its blowup, $\tilde{S}$,
of $S$ at a smooth point $p$  is also Fano. Let $\beta\in
H_2(S,\mathbb{Z})$. Then for any genus $g$, we have
\begin{equation}
    K^S_{g,\beta} = K^{\tilde{S}}_{g,p!(\beta)},
\end{equation}
where $p:\tilde{S}\longrightarrow S$ is the natural projection of
the blowup.
\end{theorem}

\begin{remark}
   Theorem \ref{thm-1-1} confirmed the Chiang-Klemm-Yau-Zaslow's
   expectation about the genus zero local Gromov-Witten invariants
   of $K_{\mathbb P}^1$ and $K_{{\mathbb F}_1}$ and generalized their
   expectation to any genus. In particular, our theorem also make
   it possible to compute the local Gromov-Witten invariants of
   nontoric del Pezzo surfaces $\tilde{\mathbb P}^2_r$, $4\leq r\leq 8$
   from the local Gromov-Witten invariants of toric del Pezzo
   surfaces $\tilde{\mathbb P}^2_r$, $1\leq r\leq 3$.
\end{remark}

\begin{remark}
    In \cite{LLW}, the authors consider the genus zero open Gromov-Witten
    invariants.
\end{remark}

{\bf Acknowledgements} The author would like to thank  Prof.
Yongbin Ruan, Wei-Ping Li and Zhenbo Qin for their valuable
discussions. The author also would like to thank M. Roth for
explaining the properties of local del Pezzo surfaces during our
visiting MPI-Bonn.

\section{Gromov-Witten invariants}

We use \cite{CK, LR} as our general reference on moduli spaces of
stable maps, Absolute/relative Gromov-Witten invariants and its
degeneration formula.

Let $X$ be a smooth complex projective manifold and $\beta\in
H_2(X,{\mathbb Z})$. Let $\overline{\mathcal M}_{g,n}(X,\beta)$ be
the moduli space of $n$-pointed stable maps $f:(\Sigma;
x_1,\cdots,x_n)$ $\longrightarrow X$ from a nodal curve $\Sigma$ with
arithmetic genus $g(\Sigma) = g$ and degree $[f(\Sigma)] =
\beta$. Let $e_i: \overline{\mathcal
M}_{g,n}(X,\beta)\longrightarrow X$ be the evaluation maps
$f\mapsto f(x_i)$. The Gromov-Witten invariant for classes
$\alpha_i\in H^*(X)$, $1\leq i\leq n$, is given by
$$
\langle \alpha_1,\cdots,\alpha_n\rangle^X_{g,n,\beta}
:=\int_{[\overline{\mathcal
M}_{g,n}(X,\beta)]^{vir}}e_1^*\alpha_1\cdots e_n^*\alpha_n.
$$

The degeneration formula \cite{LR, IP, Li} provides a rigorous
formulation about the change of Gromov-Witten invariants under the
semi-stable degeneration, or symplectic cutting. The formula
related  the absolute Gromov-Witten invariant of $X$ to the
relative Gromov-Witten invariants of two smooth pairs.

Now we recall the relative invariants of a smooth pair $(X,Z)$
with $Z\hookrightarrow X$ a smooth divisor. Let $\beta\in
H_2(X,{\mathbb Z})$ and $\mu = \{\mu_1,\cdots,\mu_{\ell(\mu)}\}\in
{\mathbb N}^{\ell(\mu)}$ be a partition of
$|\mu|:=\sum_{i=1}^{\ell(\mu)}\mu_i = \beta\cdot Z$. Let $\Gamma =
(g,n,\beta,\mu)$ be a relative graph. For $A\in H^*(X)^{\otimes
n}$ and $\delta_\mu\in H^*(Z)^{\otimes \ell(\mu)}$, the relative
invariant of stable maps with topological type $\Gamma$(i.e. with
contact order $\mu_i$ in $Z$ at the $i$-th relative point) is
$$
\langle A\mid \delta_\mu\rangle_{\Gamma}^{X,Z} :=
\int_{[\overline{\mathcal M}_\Gamma (X,Z)]^{vir}}e_X^*A\cup e^*_Z
\delta_\mu
$$
where $e_X: \overline{\mathcal M}_\Gamma(X,Z)\longrightarrow
X^n$, $e_Z: \overline{\mathcal M}_\Gamma(X,Z)\longrightarrow
Z^{\ell(\mu)}$ are evaluation maps on absolute marked points and relative marked
points respectively.

If $\Gamma = \coprod_\pi \Gamma^\pi$, the relative invariants
(with disconnceted domain curves)
$$
  \langle A\mid \delta_\mu\rangle_\Gamma^{\bullet X,Z}:=
  \prod_\pi\langle A\mid \delta_\mu \rangle_{\Gamma^\pi}^{X,Z}
$$
is defined to be the product of each connected component.

In the following, we shall discuss the degeneration formula which is the main tool employed in this paper.

Let $\pi : \chi \longrightarrow D$ be a smooth 4-fold over a
 disk $D$
such that $\chi_t = \pi^{-1}(t) \cong X $ for $t\not= 0$ and $\chi_0$ is a union of two smooth 3-folds $X_1$
and $X_2$ intersecting transversely along a smooth surface $Z$. We write $\chi_0 = X_1\cup_Z X_2$. Assume that
 $Z$ is simply connected.

Consider the natural maps
$$
  i_t: X=\chi_t \longrightarrow \chi,\,\,\,\,\,\,\,\,
  i_0:\chi_0\longrightarrow \chi,
$$
and the gluing map
$$
  g= (j_1,j_2) : X_1\coprod X_2\longrightarrow \chi_0.
$$
We have
$$
  H_2(X)\stackrel{i_{t*}}{\longrightarrow}
  H_2(\chi)\stackrel{i_{0_*}}{\longleftarrow}
  H_2(\chi_0)\stackrel{g_*}{\longleftarrow} H_2(X_1)\oplus
  H_2(X_2),
$$
where $i_{0*}$ is an isomorphism since there exists a deformation retract from $\chi$ to $\chi_0$(see \cite{C}) and
$g_*$ is surjective from Mayer-Vietoris sequence. For $\beta\in H_2(X)$, there exist $\beta_1\in H_1(X_1)$ and $\beta_2\in H_2(X_2)$
such that
$$
  i_{t*}(\beta) = i_{0_*}(j_{1_*}(\beta_1) + j_{2_*}(\beta_2)).
$$
For simplicity, we write $\beta = \beta_1+\beta_2$ instead.

Since the family $\chi\longrightarrow D$ comes from a trivial family, all cohomology classes $\alpha\in H^*(X)^{\otimes n}$ have
global liftings and the restriction $\alpha(t)$ on $\chi_t$ is defined for all $t$.

For $\{\delta_i\}$ a basis of $H^*(Z)$ with $\{\delta^i\}$ its dual basis and a partition $\mu$, denote $\delta_\mu
=\delta_{i_1}\otimes\cdots\otimes \delta_{i_{\ell(\mu)}}$ and its dual $\check{\delta}_\mu=\delta^{i_1}\otimes\cdots\otimes \delta^{i_{\ell(\mu)}}$.
The degeneration formula expresses the absolute invariants of $X$ in terms of the relative invariants of the two smooth pairs $(X_1,Z)$ and $(X_2,Z)$:
\begin{equation}
  \langle \alpha\rangle^X_{g,n, \beta} =
  \sum_\mu\sum_{\eta\in\Omega_\beta} C_\eta\langle
  j_1^*\alpha(0)\mid \delta_\mu\rangle^{\bullet
  X_1,Z}_{\Gamma_1}\langle j_2^*\alpha(0)\mid
  \check{\delta}_\mu\rangle^{\bullet X_2,Z}_{\Gamma_2}.
\end{equation}
Here $\eta = (\Gamma_1,\Gamma_2, I_{\ell(\mu)})$ is an admissible triple which consists of (possibly disconnected) topological types
$$
  \Gamma_i = \coprod _{\pi=1}^{|\Gamma_i|}\Gamma_i^\pi
$$
with the same contact order partition $\mu$ under the
identification $I_\mu$ of relative marked points. The gluing
$\Gamma_1 +_{I_{\ell{\mu}}} \Gamma_2$ has type $(g,n,\beta)$ and
is connected. In particular, $\ell(\mu)=0$ if and only if that one
of the $\Gamma_i$ is empty. The total genus $g_i$, total number of
absolute marked  points $n_i$ and the total degree $\beta_i\in
H_2(X_i)$ satisfy the splitting relations $g=
g_1+g_2+\ell(\mu)+1-|\Gamma_1| -|\Gamma_2|$, $n_1+n_2 = n$ and $
\beta_1 + \beta_2 = \beta$.

The constants $C_\eta = m(\mu)/|\mbox{Aut} \eta|$, where $m(\mu) = \prod \mu_i$ and $\mbox{Aut} \eta = \{\sigma\in S_{\ell(\mu)}\mid \eta^\sigma = \eta\}$.
We denote by $\Omega$ the equivalence class of all admissible triples, also by $\Omega_\beta$ and $\Omega_\mu$ the subset with fixed degree $\beta$
and fixed contact order $\mu$ respectively.

 For the dimensions of the related moduli spaces in the
 degeneration formula, we have

 \begin{lemma}\label{dimension} With the assumption as above,
\begin{equation}
\dim_{\mathbb C} \overline{\mathcal M}_{\Gamma_1} + \dim_{\mathbb C} \overline{\mathcal M}_{\Gamma_2} = \dim_{\mathbb C} \overline{\mathcal M}_\Gamma + 2\ell(\mu).
 \end{equation}
 \end{lemma}

\section{Projective completion  }\label{pc}
   In this section, we describe how to obtain $Y_{\tilde{S}}$ from
$Y_S$ by the degenerations. This makes it possible to find some
relations between the local Gromov-Witten invariants of
$\tilde{S}$ and $S$.

   Let $S$ be a smooth surface and $Y_S= \mathbb{P}(K_S\oplus
{\mathcal O})$ the projective completion of its canonical bundle
$K_S$. Pick a smooth point $p_0\in S$ and blow it up, then we
obtain the blowup $\tilde{S}$ of $S$ at the point $p_0$ with the
natural projection $p: \tilde{S}\longrightarrow S$ and denote by
$E$ the exceptional divisor in $\tilde{S}$. Since $Y_S$ is the
bundle $\mathbb{P}(K_S\oplus {\mathcal O})$ over $S$, one can pull
this bundle back to $\tilde{S}$ using the projection $p$. It is
easy to see that the pullback bundle is the same thing as blowing
up the fiber over $p_0$. Denote by $\tilde{Y}_S$ the blowup of
$Y_S$ along the fiber $F_{p_0}\cong \mathbb{P}^1$ over $p_0$, and
the exceptional divisor in $\tilde{Y}_S$ is denoted by $D_1:=
E\times \mathbb{P}^1 = \mathbb{P}_{\mathbb{P}^1}({\mathcal
O}\oplus {\mathcal O})$. In $\tilde{Y}_S$, take a section,
$\sigma$ , corresponding to ${\mathcal O}\longrightarrow {\mathcal
O}\oplus K_S$, of the exceptional divisor $D_1$ over $E$ and blow
it up. Denote by $Z$ the blown-up manifold,then $Z$ has a natural
projection $\pi$ to $\tilde{S}$ given by the composition of the
blowup projection $Z\longrightarrow \tilde{Y}_S$ and the bundle
projection $\tilde{Y}_S\longrightarrow \tilde{S}$. It is easy to
see that the fiber $\pi^{-1}(E)$ has two normal crossing
components: $D_1\cong \mathbb{F}_0$ and $D_2\cong \mathbb{F}_1$
intersecting along a section $\sigma$ with the normal bundle
$N_{\sigma|\mathbb{F}_0}\cong {\mathcal O}$ and
$N_{\sigma|\mathbb{F}_1}\cong {\mathcal O}(-1)$ respectively.

Next, we consider the projective completion $Y_{\tilde{S}}$. Since
the restriction $K_{\tilde{S}}\mid_E$ of the canonical bundle
$K_S$ to the exceptional divisor $E$ in $\tilde{S}$ is isomorphic
to ${\mathcal O}(-1)$, so we can pick up a section, $\sigma_1$, of
the restriction of $Y_{\tilde{S}}$ to $E$ satisfying $\sigma_1^2
=-1$. Then we blow this section $\sigma_1$ up, and it is easy to
know that the blown-up manifold is $Z$.

Let $\tilde{\mathbb P}^2_{r}$ be the blowup of ${\mathbb P}^2$ at
r points. Pick one more point $p$ and blow it up, then we obtain
$\tilde{\mathbb P}^2_{r+1} $ with the map $p:\tilde{\mathbb
P}^2_{r+1}\longrightarrow \tilde{\mathbb P}^2_r$ and denote by $E$
the exceptional divisor in $\tilde{\mathbb P}^2_{r+1}$. It is
well-known that for $0\leq r \leq 3$, $\tilde{\mathbb{P}}_r^2$ is
toric, but for $4\leq r\leq 8$, $\tilde{\mathbb P}^2_r$ is
non-toric. In \cite{CKYZ}, the authors computed the genus zero
local Gromov-Witten invariants of $\tilde{\mathbb{P}}_r^2$ with
$0\leq r \leq 3$ of lower degree. As opposed to toric del Pezzo
surfaces, one can not directly use localization with respect to a
torus action because there is no torus action on a generic del
Pezzo surface  $\tilde{\mathbb P}^2_r$,$4\leq r\leq 8$. Our
Theorem \ref{thm-1-1} implies that for some degrees, we could
compute any genus local Gromov-Witten invariants of non-toric
surfaces $\tilde{\mathbb P}^2_r$ with $4\leq r\leq 8$ from the
local Gromov-Witten invariants of $\tilde{\mathbb P}^2_r$ with
$0\leq r\leq 3$.

\section{Main theorem }
In this section, we will study the change of local Gromov-Witten
invariants under the blowup of surface $S$. Throughout this
section, we assume that $S$ and its blowup surface  all are  Fano
surfaces.

If we blowup $Y_S$ along the fiber over the point $p_0\in S$, by
the blowup formula of Gromov-Witten invariant, see Theorem1.5 of
\cite{H1}, for the genus zero invariants we have
$n_{0,\beta}^{Y_S} = n^{\tilde{Y_S}}_{0, p!(\beta)}$. In this
section, we first want to generalize this result to the case of
any genus.

\begin{lemma}\label{result-1}
Suppose that $S$ and its blowup $\tilde{S}$ are Fano surfaces. Let
$\tilde{Y}_S$ be the blowup of $Y_S$ along the fiber over $p_0\in
S$. Then for any $\beta\in H_2(S, {\mathbb Z})$, we have
$$
 n_{g,\beta}^{Y_S} = \langle \mid \emptyset\rangle^{\tilde{Y}_S,
 D_1}_{g, p!(\beta)}
$$
where $D_1= {\mathbb P}_{{\mathbb P}^1}({\mathcal O}\oplus
{\mathcal O})\cong {\mathbb P}^1\times {\mathbb P}^1$ is the
exceptional divisor in $\tilde{Y}_S$, $p!(\beta) = PD p^*
PD(\beta)$ and $p: \tilde{S}\longrightarrow S$ is the natural
projection of the blowup.

\end{lemma}
\begin{proof}
  We degenerate $Y_S$ along the fiber $F_{p_0}\cong {\mathbb P}^1$ over the blown-up point $p_0$.
  We obtain two smooth 3-folds
$$
   X_1 = \tilde{Y}_S, \hspace{1cm}  X_2 ={\mathbb P}_{{\mathbb
   P}^1}({\mathcal O}\oplus {\mathcal O}\oplus {\mathcal O})\cong
   {\mathbb P}^2\times {\mathbb P}^1,
$$
with the common divisor $D_1\cong {\mathbb P}^1\times {\mathbb
P}^1$.

Now we apply the degeneration formula to $n_{g,\beta}^{Y_S}$, then
we have
\begin{equation}\label{gluing-1}
  n_{g,\beta}^{Y_S} = \sum_\eta C_\mu \langle \mid
  \delta_\mu\rangle_{g_1,\beta_1}^{\tilde{Y}_S,D_1}\langle \mid
  \check{\delta}_\mu\rangle_{g_2,\beta_2}^{X_2,D_1},
\end{equation}
where the summation runs over all admissible configurations
$\eta=(\Gamma_1,\Gamma_2,I_{\ell(\mu)})$ and $C_\mu=m(\mu)|Aut
(\mu)|$.

Next we consider the contribution to the Gromov-Witten invariants
of each gluing component $\eta =(\Gamma_1, \Gamma_2 ,
I_{\ell(\mu)})$ where $\Gamma_1=(g_1,\beta_1,\mu)$ and
$\Gamma_2=(g_2,\beta_2,\mu)$. According to our convention, the
$X_2$-component $u^+: C^+\longrightarrow X_2$ may have many
connected components $u^+_i:C_i^+\longrightarrow X_2$,
$i=1,\cdots,l^+$. Denote by $[u^+_i]$ the homology class in $X_2$
represented by $u_i^+(C_i^+)$. Then we have
$$
  \dim_{\mathbb C}\overline{\mathcal M}_{\Gamma_2} =
  \sum_{i=1}^{l^+} C_1[u_i^+] + \ell(\mu) - \sum \mu_i,
$$
where $C_1$ is the first Chern class of $X_2$.

Let $V$ be a complex rank $r$ vector bundle over a complex
manifold $M$, and $\pi: {\mathbb P}(V)\longrightarrow M$ be the
corresponding projective bundle. Let $\xi_V$ be the first Chern
class of the tautological bundle in ${\mathbb P}(V)$. A simple
calculation shows
\begin{equation}\label{chernclass}
              C_1({\mathbb P}(V)) = \pi^* C_1(M) + \pi^*C_1(V) -
              r\xi_V.
\end{equation}

Applying (\ref{chernclass}) to $X_2$, we obtain
$$
      C_1(X_2) = \pi^*{\mathcal O}_{{\mathbb P}^1}(2) - 3\xi,
$$
where $\xi$ is the first Chern class of the tautological bundle in
$X_2$. Since the homology class $[u_i^+]$ may be decomposed into
the sum of the base class $[u_i^+]^{{\mathbb P}^1}$ and the fiber
class $[u_i^+]^f$, so we have
$$
  C_1(X_2)\cdot [u_i^+] = \pi^*{\mathcal O}_{{\mathbb P}^1}(2)\cdot
  [u_i^+]^{{\mathbb P}^1} - 3\xi\cdot [u_i^+]^f \geq 3\sum \mu_i.
$$
Therefore, we have
$$
    \dim_{\mathbb C}\overline{\mathcal M}_{\Gamma_2} \geq
    \ell(\mu) + 2\sum \mu_i.
$$

From Lemma \ref{dimension}, we have
$$
  \dim_{\mathbb C}\overline{\mathcal M}_{\Gamma_1} \leq \ell(\mu)
  - 2\sum \mu_i \leq -\sum \mu_i.
$$
This implies that for any nontrivial partition $\mu$, we have
$$
    \langle \mid \delta_\mu\rangle^{\tilde{Y}_S,D_1}_{g_1,\beta_1}
    =0.
$$
This means that the only summand with trivial partition $\mu
=\emptyset$ has the nonzero contribution to the right hand side of
(\ref{gluing-1}). Therefore we have
$$
 n_{g,\beta}^{Y_S} = \langle \mid \emptyset\rangle^{\tilde{Y}_S,
 D_1}_{g, p!(\beta)}.
$$
This completes the proof of the lemma.
\end{proof}

\begin{lemma}\label{result-2}
 For any $\beta\in H_2(S,{\mathbb Z})$, we have
 $$
  n_{g,p!(\beta)}^{\tilde{Y}_S} = \langle \mid \emptyset\rangle_{g,p!(\beta)}^{\tilde{Y}_S,D_1}.
 $$
\end{lemma}

\begin{proof}

  We degenerate $\tilde{Y}_S$ along the exceptional divisor
  $D_1\cong {\mathbb P}^1\times {\mathbb P}^1$. Then we obtain two
  smooth  $3$-folds
  $$
      X_1 = \tilde{Y}_S, \hspace{1cm}   X_2 = {\mathbb
      P}_{D_1}( N_{D_1}\oplus {\mathcal O}),
  $$
  where the normal bundle of the divisor $D_1$ in $\tilde{Y}_S$ is $N_{D_1}= {\mathcal O}(-1,-1)$.

  Similar to the proof of Lemma \ref{result-1}, applying the
  degeneration formula to $n_{g,p!(\beta)}^{\tilde{Y}_S}$, we have
  \begin{equation}\label{formula-1}
    n_{g,p!(\beta)}^{\tilde{Y}_S} = \sum_\eta C_\mu \langle \mid
    \mu\rangle_{g_1,\beta_1}^{\tilde{Y}_S, D_1}\langle \mid
    \check{\mu}\rangle_{g_2,\beta_2}^{X_2,D_1},
  \end{equation}
where the summation runs over all admissible configurations
$\eta=(\Gamma_1, \Gamma_2,I_{\ell(\mu)})$ and $C_\mu = m(\mu)|Aut
(\mu)|$.

Now we consider the contribution to
$n_{g,p!(\beta)}^{\tilde{Y}_S}$ of each gluing component $\eta =
(\Gamma_1, \Gamma_2, I_{\ell(\mu)})$ with $\Gamma_i= (g_i,\beta_i,
\mu)$, $i=1,2$. Assume that the $X_2$-component
$u^+:C^+\longrightarrow X_2$ has many connected components $u_i^+:
C_i^+\longrightarrow X_2$, $i=1,\cdots,l^+$. Denote by $[u_i^+]$
the homology class represented by $u_i^+(C_i^+)$. Therefore we
have
$$
  \dim_{\mathbb C}\overline{\mathcal M}_{\Gamma_2} =
  \sum_{i=1}^{l^+} C_1[u_i^+] + \ell(\mu) - \sum \mu_i,
$$
where $C_1$ denotes the first Chern class of $X_2$.

Note that $X_2 = {\mathbb P}_{D_1}(N_{D_1}\oplus {\mathcal O})$
and $D_1={\mathbb P}_{{\mathbb P}^1}({\mathcal O}\oplus {\mathcal
O})$. Denote by $F_p\cong {\mathbb P}^1$ the fiber of $Y_S$ at the
point $p_0$. Applying (\ref{chernclass}) to $X_2$ and $D_1$, we
obtain
\begin{eqnarray*}
               C_1(X_2) &=& \pi^*C_1(D_1) + \pi^*C_1(N_{D_1}) - 2\xi\\
               &=& \pi^*C_1(F_{p_0}) + \pi^*C_1(N_{F_p|Y_S})- 2\xi_1 +\pi^*C_1(N_{D_1})
               - 2\xi,
\end{eqnarray*}
where $\xi_1$ and $\xi$ are the first Chern classes of the
tautological bundles in ${\mathbb P}(N_{F_{p_0}|Y_S})$ and
${\mathbb P}(N_{D_1}\oplus {\mathcal O})$ respectively. Here we
denote the Chern class and its pullback by the same symbol.  It is
well-known that the normal bundle to $D_1$ in $\tilde{Y}_S$ is
just the tautological line bundle on $D_1\cong {\mathbb
P}(N_{F_{p_0}|Y_S})$. Therefore $C_1(N_{D_1}) = \xi_1$. So we have
$$
    C_1(X_2) = \pi^*C_1(F_{p_0}) - \xi_1 - 2\xi.
$$

 Note that $X_2$ is a projective bundle over $D_1$ with fiber ${\mathbb
 P}^1$.  Let $L$ be the class of a line in the fiber ${\mathbb
 P}^1$ and $e$ be the class of a line in the fiber ${\mathbb P}^1$
 in $D_1= {\mathbb P}(N_{F_{p_0}|Y_S})$. Denote by $[u_i^+]^{F_{p_0}}$ the
 homology class of the projection in $F_{p_0}$ of the curve $u_i^+$.
 Denote by $[u_i^+]^f$ the difference of $[u_i^+]$ and
 $[u_i^+]^{F_{p_0}}$, i. e. $[u_i^+]^f = [u_i^+]-[u_i^+]^{F_{p_0}}$. Then
 it is easy to know $[u_i^+]^f = aL +be$. Since $\xi\cdot [u_i^+]
 = \sum \mu_j$, where the summation runs over ends of $u_i^+$, and
 $D_1\cdot [u_i^+] =0$, so we have $\xi\cdot [u_i^+]^f = a = \sum
 \mu_j$ and $D_1\cdot [u_i^+]^f = a-b =0$. Therefore, we have
 $a=b=\sum \mu_j$. So we have $[u_i^+]^f = \sum \mu_j(L+e)$. Since
 $C_1(F_{p_0}) + C_1(N_{F_{p_0}|Y_S})=C_1(F_{p_0})\geq 0$, we have
 $$
      \sum_{i=1}^{l^+} C_1[u_i^+] \geq  4\sum\mu_i.
 $$
 Therefore we have
 $$
     \dim_{\mathbb C}\overline{\mathcal M}_{\Gamma_2} \geq
     3\sum\mu_j + \ell(\mu).
 $$

From Lemma \ref{dimension}, we have
$$
  \dim_{\mathbb C}\overline{\mathcal M}_{\Gamma_1} \leq
  \ell(\mu)-3\sum\mu_j.
$$
Therefore, for any nontrivial partition $\mu$, we have $\dim
\overline{\mathcal M}_{\Gamma_1} <0$. This implies that the only
nonzero summand in the right hand side of (\ref{formula-1}) must
have the trivial partition $\mu = \emptyset$. Therefore, we have
$$
n_{g,p!(\beta)}^{\tilde{Y}_S} = \langle \mid
\emptyset\rangle_{g,p!(\beta)}^{\tilde{Y}_S,D_1}.
$$
This proves the lemma.

\end{proof}

Summarizing Lemma \ref{result-1} and Lemma \ref{result-2}, we have
\begin{theorem}\label{result-3}
 $$
            n_{g,\beta}^{Y_S} = n_{g, p!(\beta)}^{\tilde{Y}_S}.
$$
\end{theorem}

Next, we want to compare the Gromov-Witten invariants
$n_{g,p!(\beta)}^{\tilde{Y}_S}$ of $\tilde{Y}_S$ to the
Gromov-Witten invariants of $Z $. In fact, we have

\begin{theorem}\label{result-4}
$$
   n_{g,p!(\beta)}^{\tilde{Y}_S} = n_{g,p!(\beta)}^Z.
$$
\end{theorem}

\begin{proof}
In $\tilde{Y}_S$, take a section $\sigma$ of the exceptional
divisor $D_1$ over the old exceptional divisor $E$, then
$\sigma\cong {\mathbb P}^1$ and the normal bundle to $\sigma$ in
$\tilde{Y}_S$ is $N_{\sigma\mid \tilde{Y}_S}={\mathcal
O}_{{\mathbb P}^1}(-1) \oplus {\mathcal O}$. We degenerate
$\tilde{Y}_S$ along this section $\sigma$, then we obtain two
smooth 3-folds
$$
  X_1 = Z, \hspace{1cm} X_2= {\mathbb P}_{\sigma}({\mathcal
  O}_{{\mathbb P}^1}(-1)\oplus {\mathcal O}\oplus {\mathcal O}),
$$
with the Hirzebruch surface ${\mathbb F}_1= {\mathbb
P}_{\sigma}({\mathcal O}(-1)\oplus {\mathcal O})$ as the common
divisor.

Applying the degeneration formula to
$n_{g,p!(\beta)}^{\tilde{Y}_S}$, then we have
\begin{equation}\label{gluing-3}
n_{g,p!(\beta)}^{\tilde{Y}_S} = \sum_\eta C_\mu\langle \mid
\delta_\mu\rangle_{g_1, \beta_1}^{Z,{\mathbb F}_1} \langle \mid
\check{\delta}_\mu\rangle_{g_2,\beta_2}^{X_2,{\mathbb F}_1},
\end{equation}
where the summation runs over all admissible configurations $\eta
= (\Gamma_1,\Gamma_2,I_{\ell(\mu)})$ and $C_\mu = m(\mu)
|Aut(\mu)|$.

Similar to the proof of Lemma \ref{result-1}, we consider the
contribution to $n_{g,p!(\beta)}^{\tilde{Y}_r}$ of each gluing
component $\eta = (\Gamma_1, \Gamma_2,I_{\ell(\mu)})$. Assume that
the $X_2$-component $u^+:C^+\longrightarrow X_2$ has $l^+$
components $u_i^+:C_i^+\longrightarrow X_2$, $i=1,\cdots, l^+$.
Denote by $[u_i^+]$ the homology class in $X_2$ represented by
$u_i^+(C_i^+)$. Then we have
$$
  \dim_{\mathbb C}\overline{\mathcal M}_{\Gamma_2} =
  \sum_{i=1}^{l^+} C_1[u_i^+] + \ell(\mu) - \sum \mu_i,
$$
where  $C_1$ is the first Chern class of $X_2$. From
(\ref{chernclass}), it is easy to know
$$
    C_1(X_2) = \pi^*C_1({\mathcal O}_\sigma(1)) - 3\xi,
$$
where $\xi$ is the first Chern class of the tautological line
bundle over $X_2$. Therefore we have
$$
    \sum_{i=1}^{l^+}C_1[u_i^+] \geq 3 \sum \mu_i.
$$
Therefore, we have
$$
   \dim_{\mathbb C}\overline{\mathcal M}_{\Gamma_2} \geq \ell(\mu)
   + 2 \sum \mu_i.
$$

From Lemma \ref{dimension}, we have
$$
     \dim_{\mathbb C}\overline{\mathcal M}_{\Gamma_1} \leq
     \ell(\mu) - 2 \sum \mu_i\leq -\sum \mu_i.
$$
This means that for any nontrivial partition $\mu$, $\dim_{\mathbb
C}\overline{\mathcal M}_{\Gamma_1} <0$. This implies that the only
nonzero summand in the right hand side of (\ref{gluing-3}) must be
the trivial partition $\mu =\emptyset$. Therefore, we have
$$
   n_{g,p!(\beta)}^{\tilde{Y}_S} = \langle \mid
   \emptyset\rangle_{g,p!(\beta)}^{Z, {\mathbb F}_1}.
$$

Now it remains to prove
$$
     n_{g,p!(\beta)}^Z = \langle \mid
   \emptyset\rangle_{g,p!(\beta)}^{Z, {\mathbb F}_1}.
$$

To prove this, we degenerate $Z$ along the exceptional divisor
${\mathbb F}_1$. Then we obtain two smooth $3$-folds
$$
  X_1 = Z, \hspace{1cm} X_2 = {\mathbb P}_{{\mathbb F}_1}(N_{{\mathbb
  F}_1}\oplus {\mathcal O}),
$$
intersecting along the exceptional divisor ${\mathbb F}_1$ in $Z$
and the infinite section of the ${\mathbb P}^1$-bundle $X_2$.

 Applying the degenerate formula to
$n_{g,p!(\beta)}^Z$ , we have
\begin{equation}\label{gluing-4}
n_{g,p!(\beta)}^Z = \sum_\eta C_\mu\langle \mid
\delta_\mu\rangle_{g_1, \beta_1}^{Z,{\mathbb F}_1} \langle \mid
\check{\delta}_\mu\rangle_{g_2,\beta_2}^{X_2,{\mathbb F}_1},
\end{equation}
where the summation runs over all admissible configurations $\eta
= (\Gamma_1,\Gamma_2,T_{\ell(\mu)})$ and $C_\mu = m(\mu)
|Aut(\mu)|$.

Similar to the proof of Lemma \ref{result-2}, we consider the
contribution to $n_{g,p!(\beta)}^Z$ of each gluing component $\eta
= (\Gamma_1,\Gamma_2,I_{\ell(\mu)})$ with $\Gamma_i =
(g_i,\beta_i,\mu)$, $i = 1, 2$. Assume that the $X_2$-component
$u^+:C^+\longrightarrow X_2$ has many connected components
$u_i^+:C_i^+\longrightarrow X_2$, $i=1,\cdots,l^+$. Denote by
$[u_i^+]$ the homology class represented by $u_i^+(C_i^+)$.
Therefore we have
$$
  \dim_{\mathbb C}\overline{\mathcal M}_{\Gamma_2} =
  \sum_{i=1}^{l^+} C_1[u_i^+] + \ell(\mu) - \sum \mu_i,
$$
where $C_1$ denotes the first Chern class of $X_2$.

Note that $X_2 = {\mathbb P}_{{\mathbb F}_1}(N_{{\mathbb
F}_1}\oplus {\mathcal O})$ and ${\mathbb F}_1 = {\mathbb
P}_{\sigma}({\mathcal O}(-1)\oplus {\mathcal O})$. Applying
(\ref{chernclass}) to $X_2$ and ${\mathbb F}_1$, we obtain
\begin{eqnarray*}
   C_1(X_2) & = & \pi^*C_1({\mathbb F}_1) +\pi^* C_1(N_{{\mathbb F}_1}) -
   2\xi\\
   &=& \pi^*C_1({\mathcal O}_\sigma(1) - \xi_1 -2\xi,
\end{eqnarray*}
where $\xi_1$ and $\xi$ are the first Chern classes of the
tautological bundles in ${\mathbb P}_\sigma ({\mathcal
O}(-1)\oplus {\mathcal O})$ and ${\mathbb P}(N_{{\mathbb
F}_1}\oplus {\mathcal O})$ respectively. Here we denote the Chern
class and its pullback by the same symbol. The same calculation as
in the proof of Lemma \ref{result-2} shows that
$$
   \sum_{i=1}^{l^+} C_1[u_i^+] \geq 4\sum \mu_i.
$$
Therefore we have
$$
\dim_{\mathbb C}\overline{\mathcal M}_{\Gamma_2} \geq 3\sum\mu_i
+\ell(\mu).
$$

From Lemma \ref{dimension}, we have
$$
   \dim_{\mathbb C}\overline{\mathcal M}_{\Gamma_1} \leq \ell(\mu)
   - 3\sum \mu_j.
$$
Therefore, for any nontrivial partition $\mu$, we have
$\dim\overline{\mathcal M}_{\Gamma_1}<0$. This implies that the
only nonzero summand in the right hand side of (\ref{gluing-4})
must have the trivial partition $\mu=\emptyset$. Therefore we have
$$
   n_{g,p!(\beta)}^Z = \langle \mid
   \emptyset\rangle_{g,p!(\beta)}^{Z,{\mathbb F}_1}.
$$
This proves the theorem.
\end{proof}

Finally, we want to prove that $n_{g,p!(\beta)}^{Y_{\tilde{S}}}=
n_{g,p!(\beta)}^Z$.

\begin{theorem}\label{result-5}
$$
n_{g,p!(\beta)}^{Y_{\tilde{S}}}= n_{g,p!(\beta)}^Z
$$
\end{theorem}

\begin{proof}
   Take a section $\sigma_1\cong {\mathbb P}^1$ of $Y_{\tilde{S}}\mid_E\cong {\mathbb F_1}$
such that $\sigma_1^2=-1$. Then we degenerate $Y_{\tilde{S}}$
along the section $\sigma_1$. Then we obtain two $3$-folds, see
Section \ref{pc},
$$
      X_1=Z, \hspace{1cm} X_2 = {\mathbb P}_{\sigma_1}({\mathcal
      O}(-1)\oplus {\mathcal O}(-1)\oplus {\mathcal O}),
$$
with the common divisor ${\mathbb F}_0= {\mathbb
P}_{\sigma_1}({\mathcal O}(-1)\oplus {\mathcal O}(-1))$.

Applying the degenerate formula to
$n_{g,p!(\beta)}^{Y_{\tilde{S}}}$, we have
\begin{equation}\label{gluing-5}
n_{g,p!(\beta)}^{Y_{\tilde{S}}} = \sum_\eta C_\mu\langle \mid
\delta_\mu\rangle_{g_1, \beta_1}^{Z,{\mathbb F}_0} \langle \mid
\check{\delta}_\mu\rangle_{g_2,\beta_2}^{X_2,{\mathbb F}_0},
\end{equation}
where the summation runs over all admissible configurations $\eta
= (\Gamma_1,\Gamma_2,T_{\ell(\mu)})$ and $C_\mu = m(\mu)
|Aut(\mu)|$.

Similar to the proof of Lemma \ref{result-1}, we need to prove
that the summand with nonzero contribution in the right hand side
of (\ref{gluing-5}) must have trivial partition $\mu = \emptyset$.
Using the same notation as before, we have
$$
  \dim_{\mathbb C}\overline{\mathcal M}_{\Gamma_2} =
  \sum_{i=1}^{l^+} C_1[u_i^+] + \ell(\mu) - \sum \mu_i,
$$
where $C_1$ denotes the first Chern class of $X_2$.

Note that $X_2 = {\mathbb P}_{\sigma_1}({\mathcal O}(-1)\oplus
{\mathcal O}(-1)\oplus {\mathcal O})$. From (\ref{chernclass}), it
is easy to know
\begin{equation*}
    C_1(X_2) = \pi^*C_1(\sigma_1) + \pi^*C_1({\mathcal O}(-1)\oplus
    {\mathcal O}(-1)) -3\xi = -3\xi,
\end{equation*}
where $\xi$ is the first Chern class of the tautological line
bundle over $X_2$. Therefore we have
$$
   \dim_{\mathbb C}\overline{\mathcal M}_{\Gamma_2} = \ell(\mu) +
   2\sum\mu_j.
$$

From Lemma \ref{dimension}, we have
$$
   \dim_{\mathbb C}\overline{\mathcal M}_{\Gamma_1} = \ell(\mu)
   -2\sum \mu_j \leq - \sum\mu_j.
$$
This means that for any nontrivial partition $\mu$, $\dim_{\mathbb
C}\overline{\mathcal M}_{\Gamma_1} <0$. This implies that the only
nonzero summand in the right hand of (\ref{gluing-5}) must have
the trivial partition $\mu=\emptyset$. Therefore we have
$$
    n_{g,p!(\beta)}^{Y_{\tilde{S}}} = \langle \mid
    \emptyset\rangle_{g,p!(\beta)}^{Z,{\mathbb F}_0}.
$$

Now it remains to prove
\begin{equation}\label{lastone}
 n_{g,p!(\beta)}^Z = \langle\mid
 \emptyset\rangle_{g,p!(\beta)}^{Z,{\mathbb F}_0}.
\end{equation}

To prove this, we degenerate $Z$ along the exceptional divisor
${\mathbb F}_0$. Then we obtain two $3$-folds
$$
     X_1= Z, \hspace{1cm} X_2= {\mathbb P}_{{\mathbb
     F}_0}(N_{{\mathbb F}_0}\oplus {\mathcal O}).
$$

Similar to the proof of Theorem \ref{result-4}, we consider the
contribution of each gluing component $\eta=
(\Gamma_1,\Gamma_2,I_{\ell(\mu)})$. Using the same notation as
before, we have
$$
\dim_{\mathbb C}\overline{\mathcal M}_{\Gamma_2} =
\sum_{i=1}^{l^+} C_1[u_i^+] +\ell(\mu) - \sum\mu_j,
$$
where $C_1$ denotes the first Chern class of $X_2$.

Note that $X_2= {\mathbb P}_{{\mathbb F}_0}(N_{{\mathbb
F}_0}\oplus {\mathcal O})$ and ${\mathbb F}_0= {\mathbb
P}_{\sigma_1}({\mathcal O}(-1)\oplus {\mathcal O}(-1)$. Applying
(\ref{chernclass}) to $X_2$ and ${\mathbb F}_0$, we have
\begin{eqnarray*}
    c_1(X_2) &=& \pi^*C_1({\mathbb F}_0) + \pi^*C_1(N_{{\mathbb F}_0})
    -2\xi\\
    &=& \pi^*C_1(\sigma_1) + \pi^*C_1({\mathcal O}(-1)\oplus
    {\mathcal O}(-1)) -2\xi_1 + \pi^*C_1(N_{{\mathbb F}_0})
    -3\xi\\
    &=& -\xi_1 -2\xi,
\end{eqnarray*}
where $\xi_1$ and $\xi$ are the first Chern classes of the
tautological bundles in ${\mathbb P}_{\sigma_1} ({\mathcal
O}(-1)\oplus {\mathcal O}(-1))$ and ${\mathbb P}(N_{{\mathbb
F}_0}\oplus {\mathcal O})$ respectively. The same calculation as
in the proof of Lemma \ref{result-2} shows that
$$
   \dim_{\mathbb C}\overline{\mathcal M}_{\Gamma_2} = \ell(\mu)+ 3\sum\mu_j.
$$

From Lemma \ref{dimension}, we have
$$
   \dim_{\mathbb C}\overline{\mathcal M}_{\Gamma_1} =
   \ell(\mu)-3\sum\mu_j.
$$
As before, this implies (\ref{lastone}). This completes the proof
of the theorem.

\end{proof}

\begin{remark}
From (\ref{local=GW}), Theorem \ref{result-4} and Theorem
\ref{result-5}, it is easy to know that Theorem \ref{thm-1-1}
holds.
\end{remark}

\end{document}